\documentclass[letterpaper, 10 pt, conference]{ieeeconf}  

\IEEEoverridecommandlockouts                              
                                                          
\usepackage{amsmath}
\usepackage{amssymb}
\usepackage{mathtools}

\DeclareMathOperator{\Rank}{rank}

\newtheorem{prop}{Proposition}
\newtheorem{defn}{Definition}
\newtheorem{rem}{Remark}

\title{\LARGE \bf Flatness-Based Controller Design for Backward-Flat Systems}

\author{Johannes Schrotshamer, Bernd Kolar, Markus Sch{\"o}berl
\thanks{This research was funded in whole, or in part, by the Austrian Science Fund (FWF) P36473. For the purpose of open access, the author has applied a CC BY public copyright license to any Author Accepted Manuscript version arising from this submission. All authors are with the Institute of Control Systems, Johannes Kepler University Linz, Altenbergerstrasse 69, 4040 Linz, Austria.\newline {Email: \tt\small johannes.schrotshamer@jku.at, \newline bernd.kolar@jku.at, markus.schoeberl@jku.at}}}

\begin{document}

\maketitle
\thispagestyle{empty}
\pagestyle{empty}

\begin{abstract}                
This paper addresses the design of flatness-based tracking controllers for backward-flat nonlinear discrete-time systems. First, we show that backward-flat discrete-time systems can be constructively obtained from differentially flat continuous-time systems by applying an implicit Euler discretization to a structurally flat triangular representation. Subsequently, two approaches for the exact linearization of backward-flat systems are considered. Besides a dynamic feedback based on prelongations, we show that a linearizing feedback involving lower-order backward-shifts of the flat output can be employed without explicitly implementing the corresponding dynamic extension. This allows the order of the resulting tracking error dynamics to be reduced while requiring only stored values of past system trajectories. Based on the resulting linear input-output representation, flatness-based tracking control laws are derived. The proposed discretization and controller design are illustrated for a 2D gantry crane.
\end{abstract}



\section{Introduction}
The concept of differential flatness was introduced in the 1990s by Fliess, L{\'e}vine, Martin, and Rouchon for nonlinear continuous-time systems (see, e.g., \cite{FliessLevineMartinRouchon:1995,FliessLevineMartinRouchon:1999}) and has since become an important tool in nonlinear control theory.
A nonlinear continuous-time system
\begin{equation}\label{eq:sys_cont}
    \dot{x}=f(x,u),
\end{equation}
with an $n$-dimensional state $x$, $m$ inputs $u$, and smooth functions $f^i(x,u)$ is called flat if its solutions $(x(t),u(t))$ are in one-to-one correspondence with the solutions $y(t)$ of a trivial system (arbitrary trajectories that need not satisfy any differential equation).
In other words, this one-to-one correspondence implies the existence of so-called flat outputs $y=\varphi(x,u,\dot{u},\ldots,u^{(q)})$ such that the system variables $x$ and $u$ can be parameterized by $y$ and a finite number of its time derivatives, i.e.,
\begin{equation*}
  (x,u) = F(y,\dot{y},\ldots,y^{(r)}),
\end{equation*}
where $F$ denotes the parameterizing map with respect to the flat output $y$.
These properties considerably simplify motion planning and tracking controller design, which has led to successful applications in a wide range of practical problems.

For practical implementation, however, the continuous-time control law must be evaluated at a sufficiently high sampling rate. If this cannot be ensured, it has been shown, e.g., in \cite{DiwoldKolarSchoberl:2022}, that it can be advantageous to design a discrete-time flatness-based controller directly for a suitable discretization of the continuous-time system.
Regarding the flatness of nonlinear discrete-time systems, known as difference flatness, a concise geometric definition was introduced in \cite{DiwoldKolarSchoberl:2020}. Analogously to the continuous-time case, this definition is based solely on the existence of a one-to-one correspondence between the solutions of a flat system and those of a trivial system. Within this notion, both the flat output and the flat parameterization may depend on forward- as well as backward-shifts; see also \cite{GuillotMillerioux2020,Kaldmae:2021}. To the best of our knowledge, this represents the most general definition of difference flatness available in the literature. Although general verifiable necessary and sufficient conditions are currently not available for this class of systems, flatness-based tracking controllers have been developed in \cite{KolarDiwoldGstottnerSchoeberl:2022}.

As discussed in \cite{SchrotshamerKolarSchoeberl:2025}, this general notion of difference flatness contains forward-flat and backward-flat systems as important subclasses. Forward-flat systems, which have been studied, e.g., in \cite{Sira-RamirezAgrawal:2004,KaldmaeKotta:2013,KolarKaldmaeSchoberlKottaSchlacher:2016}, are already supported by an extensive literature including necessary and sufficient conditions established in \cite{KolarDiwoldSchoberl:2023,KolarSchrotshamerSchoberl:2025} and tracking control methods developed in \cite{DiwoldKolarSchoberl:2022}. Furthermore, it has been shown in \cite{DiwoldKolarSchoberl:2022} that a flat continuous-time system \eqref{eq:sys_cont} can be discretized in a flatness-preserving manner by first transforming it into a structurally flat triangular form and subsequently applying an explicit Euler discretization. The resulting discrete-time system is forward-flat, whereas the flatness property is, in general, not preserved under exact or approximate discretizations. In contrast, the literature on backward-flat systems is far less developed, where a systematic approach for testing backward-flatness was only recently introduced in \cite{SchrotshamerKolarSchoeberl:2026}.

To address this gap, in this contribution we first show that backward-flat systems can be constructively derived from differentially flat continuous-time systems \eqref{eq:sys_cont} through a flatness-preserving discretization based on the implicit Euler method, building on the results of \cite{DiwoldKolarSchoberl:2023}. Moreover, as the main contribution of this paper, we demonstrate the design of tracking control laws for backward-flat systems by first exactly linearizing the corresponding system equations and subsequently applying an additional feedback to stabilize the reference trajectories.

Since backward-flat systems form a subclass of flat systems in the general sense, it may seem natural to directly apply the results of \cite{KolarDiwoldGstottnerSchoeberl:2022} for deriving a tracking controller for backward-flat systems.
However, the main focus of \cite{KolarDiwoldGstottnerSchoeberl:2022} is on reducing the order of the error dynamics by lowering the order of the forward-shifts of the flat output used in the construction of the linearizing feedback. The analogous possibility of using lower-order backward-shifts of the flat output for the linearizing feedback and the subsequent tracking controller design was not considered. Consequently, for backward-flat systems, the order of the resulting tracking error dynamics is determined by the highest-order backward-shifts occurring in the parameterization of the system variables.
In the present contribution, we show that backward-flat systems can be exactly linearized either by means of a dynamic feedback based on the results of \cite{SchrotshamerKolarSchoeberl:2026b} or by using lower-order backward-shifts of the flat output, with the latter approach reducing the order of the resulting tracking error dynamics.

The paper is organized as follows: In Section~\ref{sec:two}, we recall the notions of difference flatness and backward-flatness and show that the implicit Euler discretization of a continuous-time system in structurally flat triangular form yields a backward-flat discrete-time system. Section~\ref{sec:three} addresses the controller design. We first discuss the exact linearization of backward-flat systems, either by an endogenous dynamic feedback in the form of prelongations or by a generalized state transformation based on stored past values of the system trajectory. We then derive tracking control laws that lead to error dynamics of different orders. In Section~\ref{sec:four}, the results are illustrated by the example of a 2D gantry crane.

\paragraph*{Notation}
    To enhance readability, we commonly use the superscript $+$ to indicate a one-fold forward-shift of a variable, whereas higher-order forward-shifts and backward-shifts are written using bracketed subscripts.
    Thus, for $\alpha \geq 0$, $u_{[\alpha]}$ denotes the $\alpha$-th forward-shift of the variable $u$, and $u_{[-\alpha]}$ its $\alpha$-th backward-shift.
    The corresponding forward- and backward-shift operators, denoted by $\delta^{\alpha}$ and $\delta^{-\alpha}$, act on variables as well as on functions; see, e.g., \cite{DiwoldKolarSchoberl:2020} for details.
    If a variable consists of several components, i.e., $y = (y^1,\ldots,y^m)$, we write $y_{[\alpha]} = (y^1_{[\alpha]},\ldots,y^m_{[\alpha]})$.\\
    Shifts of different orders for the individual components are represented compactly by multi-indices. 
    For two multi-indices $S_1 = (s_{1,1},\ldots,s_{1,m})$ and $S_2 = (s_{2,1},\ldots,s_{2,m})$, $y_{[-S_1]}$ and $y_{[S_2]}$ denote the corresponding backward- and forward-shifts of the individual components of $y$. Furthermore, $y_{[-S_1,S_2]}$ is an abbreviation for $y$ together with its backward-shifts up to the orders $S_1$ and its forward-shifts up to the orders $S_2$.
    We write $S_1 \leq S_2$ if $s_{1,j} \leq s_{2,j}$ holds for all $j=1,\ldots,m$, and $y^j_{[s_{1,j},s_{2,j}]}$ is defined to be empty whenever $s_{1,j} > s_{2,j}$.
    Addition and subtraction of two multi-indices $S_1$ and $S_2$, as well as of a multi-index $S_1$ and a scalar $a$, are performed component-wise, i.e.,
    \begin{equation*}
      \begin{aligned}
        S_1 \pm S_2 &= (s_{1,1} \pm s_{2,1},\ldots,s_{1,m} \pm s_{2,m}), \\
        S_1 \pm a   &= (s_{1,1} \pm a,\ldots,s_{1,m} \pm a)\,.
      \end{aligned}
    \end{equation*}
    The sum over all components of a multi-index $S_1$ is denoted by
    \begin{equation*}
      \#S_1 = \sum_{j=1}^{m} s_{1,j}\,.
    \end{equation*}

\section{Discrete-Time Systems and Flatness} \label{sec:two}

    \subsection{Flatness of discrete-time systems}

    In this contribution, we consider nonlinear time-invariant discrete-time systems of the form
    \begin{equation}\label{eq:sysEq}
        x^{i,+}=f^{i}(x,u)\,,\quad i=1,\dots,n
    \end{equation}
    with $\dim(x)=n$, $\dim(u)=m$, and smooth functions $f^{i}(x,u)$. Additionally, we assume that the system equations \eqref{eq:sysEq} satisfy the submersivity condition $\Rank(\partial_{(x,u)}f)=n$, which is a common assumption in the discrete-time literature, as it is necessary for accessibility and consequently also for flatness.
    Moreover, due to this submersivity condition, there always exist $m$ functions $g(x, u)$ such that the map
    \begin{equation*}\label{eq:sysEq_ext}
        x^{+} = f(x, u)\,,\quad \zeta = g(x, u)
    \end{equation*}
    is locally a diffeomorphism, with its inverse given by
    \begin{equation}\label{eq:sysEq_ext_inv}
        x=\psi_x(x^+,\zeta)\,, \quad u=\psi_u(x^+,\zeta)\,.
    \end{equation}
    Following \cite{DiwoldKolarSchoberl:2020}, we call a discrete-time system \eqref{eq:sysEq} flat if there exists a one-to-one correspondence between its solutions $(x(k),u(k))$ and solutions $y(k)$ of a trivial system (arbitrary trajectories that need not satisfy any difference equation) with the same number of inputs. As formally stated in Definition 1 of \cite{DiwoldKolarSchoberl:2020}, the one-to-one correspondence implies that the values of $x(k)$ and $u(k)$ at a given time step may depend on an arbitrary but finite number of past and future values of $y(k)$, and conversely, $y(k)$ may depend on an arbitrary but finite number of past and future values of $x(k)$ and $u(k)$. However, throughout this contribution, we use the slightly modified yet equivalent formulation as stated in \cite{SchrotshamerKolarSchoeberl:2025}.
    \vspace{0.05cm}

	\begin{defn}
		\label{def:flatness}The system \eqref{eq:sysEq} is said to be flat around an equilibrium $(x_{0},u_{0})$, if the $n+m$ coordinate functions $x$ and $u$ can be expressed locally by an $m$-tuple of functions
		\begin{equation}
			y^{j}=\varphi^{j}(\zeta_{[-Q_1,-1]},x,u,u_{[1,Q_2]})\,,\quad\,j=1,\ldots,m\label{eq:flat_output}
		\end{equation}
		called the flat output, and its backward- and forward-shifts up to some finite order, i.e.
    	\begin{equation}\label{eq:flat_param}
    		\begin{aligned}
    			x&=F_{x}(y_{[-R_1,R_2-1]})\,, &\; u&=F_{u}(y_{[-R_1,R_2]})\,.
    		\end{aligned}
    	\end{equation}
        \vspace{-0.3cm}
	\end{defn}
    For a given flat output \eqref{eq:flat_output}, the parameterization \eqref{eq:flat_param} of the state- and input variables $x$ and $u$ is unique.\\ 
    As stated in \cite{SchrotshamerKolarSchoeberl:2025}, this general definition of difference flatness gives rise to two important special cases: forward-flatness (which we do not cover in this paper) and backward-flatness. The latter arises naturally when no forward-shifts appear in the flat output \eqref{eq:flat_output} and $R_2=(0,\ldots,0)$ holds for the parameterization \eqref{eq:flat_param}.
    \vspace{0.05cm}
	\begin{defn}
		\label{def:backward-flatness}The system \eqref{eq:sysEq} is said
		to be backward-flat around an equilibrium $(x_{0},u_{0})$, if the
		$n+m$ coordinate functions $x$ and $u$ can be expressed locally
		by an $m$-tuple of functions
		\begin{equation*}
			y^{j}=\varphi^{j}(\zeta_{[-Q_1,-1]},x,u)\,,\quad j=1,\ldots,m\label{eq:backward-flat_output}
		\end{equation*}
		called the backward-flat output, and its backward-shifts up to some finite order, i.e.
        \begin{equation}\label{eq:backward_flat_param}
    		\begin{aligned}
    			x&=F_{x}(y_{[-R_1,-1]})\,, &\; u&=F_{u}(y_{[-R_1,0]})\,.
    		\end{aligned}
	    \end{equation}
        \vspace{-0.3cm}
	\end{defn}

    \subsection{Implicit Euler Discretization and Backward-Flatness}\label{subsec:flatness_preserving_disc}
    In this following, we show that backward-flat systems can be constructively derived from differentially flat continuous-time systems \eqref{eq:sys_cont} via a flatness-preserving discretization based on the implicit Euler method, building upon the results of \cite{DiwoldKolarSchoberl:2023}. Therein, the authors employ flat continuous-time systems in so called structurally flat triangular form        
    \begin{align}
    \begin{aligned}\dot{{x}}_{p} & =f_{p}({x}_{p},{x}_{p-1},{u}_{p-1})\\
    \dot{{x}}_{p-1} & =f_{p-1}({x}_{p},{x}_{p-1},{x}_{p-2},{u}_{p-1},{u}_{p-2})\\
    & \phantom{a}\vdots\\
    \dot{{x}}_{2} & =f_{2}({x}_{p},\dots,{x}_{1},{u}_{p-1},\dots,{u}_{1})\\
    \dot{{x}}_{1} & =f_{1}({x}_{p},\dots,{x}_{1},{u}_{p-1},\dots,{u}_{1},{u}_{0})
    \end{aligned}
    \label{eq:structurally_flat_triangular_form_continuous_time_system}
    \end{align}
    with the state
    \begin{equation}
    \label{eq:def_x_blocks_struct_flat_triangular_form_cont}
    \begin{aligned}
    {x} & =({x}_{p},{x}_{p-1}\dots,{x}_{1})\\
    & =(y_{p},(y_{p-1},\hat{x}_{p-1}),\dots,(y_{1},\hat{x}_{1}))
    \end{aligned}
    \end{equation}
    and the input ${u}=({u}_{p-1},\dots,{u}_{0})$ for some $p>0$ (the components $y_{k}$ and/or ${u}_{k}$ might also be empty). Due to the rank conditions
    \begin{align*}
    \begin{aligned}\mathrm{rank}(\partial_{{u}_{0}}f_{1}) & =\dim(f_{1})=\dim({u}_{0})\\
    \textrm{\ensuremath{\mathrm{rank}}}(\partial_{(\hat{x}_{k},{u}_{k})}f_{k+1}) & =\dim(f_{k+1})=\dim((\hat{x}_{k},{u}_{k}))
    \end{aligned}
    \label{eq:rank_cond}
    \end{align*}
    for $k=1,\dots,p-1$ and its triangular structure, for smooth trajectories $y(t)=(y_p(t), \ldots, y_1(t))$, the corresponding trajectories $(x(t), u(t))$ can be determined systematically by successively parameterizing the variables $(x,u)$ in terms of $y$ and its time derivatives, proceeding through the equations \eqref{eq:structurally_flat_triangular_form_continuous_time_system} from top to bottom.
    As stated in \cite{DiwoldKolarSchoberl:2023}, applying the implicit Euler discretization
    \begin{equation*}\label{eq:implicit_euler_discretization}
        x^+=x+T_s f(x^+,u)
    \end{equation*}
    to a system in structurally flat triangular form \eqref{eq:structurally_flat_triangular_form_continuous_time_system} yields a discrete-time system
    \begin{align}\label{eq:impl_disk}
    \begin{aligned}{{x}}^+_{p} & =x_{p}+T_sf_{p}({x}^+_{p},{x}^+_{p-1},{u}_{p-1})\\
    {{x}}^+_{p-1} & =x_{p-1}+T_sf_{p-1}({x}^+_{p},{x}^+_{p-1},{x}^+_{p-2},{u}_{p-1},{u}_{p-2})\\
    & \phantom{a}\vdots\\
    {{x}}^+_{2} & =x_{2}+T_sf_{2}({x}^+_{p},\dots,{x}^+_{1},{u}_{p-1},\dots,{u}_{1})\\
    {{x}}^+_{1} & =x_1+T_sf_{1}({x}^+_{p},\dots,{x}^+_{1},{u}_{p-1},\dots,{u}_{1},{u}_{0})
    \end{aligned}
    \end{align}
    that, together with \eqref{eq:def_x_blocks_struct_flat_triangular_form_cont} and
    \begin{align}
    \begin{aligned}\mathrm{rank}(\partial_{{u}_{0}}f_{1}) & =\dim(f_{1})=\dim({u}_{0})\\
    \textrm{\ensuremath{\mathrm{rank}}}(\partial_{(\hat{x}_{k}^+,{u}_{k})}f_{k+1}) & =\dim(f_{k+1})=\dim((\hat{x}_{k}^+,{u}_{k}))\,,
    \end{aligned}
    \label{eq:rank_cond_disc}
    \end{align}
    still exhibits a triangular structure. Thus, similar to the continuous-time case, the variables $(x,u)$ can once again be parameterized successively by proceeding through the equations \eqref{eq:impl_disk} from top to bottom, using \eqref{eq:def_x_blocks_struct_flat_triangular_form_cont}.
    However, the system trajectories $(x(k),u(k))$ depend on both backward- and forward-shifts of $y=(y_p, \ldots, y_1)$. Consequently, the authors of \cite{DiwoldKolarSchoberl:2023} stated that if a continuous-time system \eqref{eq:sys_cont} can be transformed into a structurally flat triangular form \eqref{eq:structurally_flat_triangular_form_continuous_time_system} via a state- and input transformation
    \begin{subequations}\label{eq:trafo_cont}
        \begin{align}
            \bar x &= \Phi_x(x) \label{subeqn:disk_trafo_x}  \\
            \bar u &= \Phi_u(x,u) \label{subeqn:disk_trafo_u}\,,
        \end{align}
    \end{subequations}
    then the implicit Euler discretization of the system $\dot{\bar{x}}=f(\bar{x}, \bar{u})$ is flat according to Definition \ref{def:flatness}. Moreover, since the input trajectories of the original and transformed input are related via the map $\bar{u}=\Phi_u(x^{+}, u)$, the implicit Euler discretization of the system $\dot{\bar{x}}=f(\bar{x}, u)$, obtained by applying only the state transformation \eqref{subeqn:disk_trafo_x}, is also flat.
    
    We propose that by using $\tilde y=(\tilde y_p, \ldots, \tilde y_1)=(y_p^+, \ldots, y_1^+)$ as the flat output instead, an even stronger conclusion can be drawn from the results in \cite{DiwoldKolarSchoberl:2023}.
    \begin{prop}
        If a continuous-time system \eqref{eq:sys_cont} can be transformed into a structurally flat triangular form \eqref{eq:structurally_flat_triangular_form_continuous_time_system} via a state- and input transformation \eqref{eq:trafo_cont}, then the implicit Euler discretization of the system $\dot{\bar{x}}=f(\bar{x}, u)$ is backward-flat according to Definition \ref{def:backward-flatness}.
    \end{prop}

    \begin{proof}
        By defining $\tilde y=(y_p^+, \ldots, y_1^+)$ as the new flat output, the backward-shifts $\tilde y_{p,[-1]}, \tilde y_{p-1,[-1]},\dots,\tilde y_{1,[-1]}$ now represent components of the state, i.e.,
        \begin{equation*}
            \begin{aligned}
            {x} & =({x}_{p},{x}_{p-1}\dots,{x}_{1})\\
            & =(\tilde y_{p,[-1]},(\tilde y_{p-1,[-1]},\hat{x}_{p-1}),\dots,(\tilde y_{1,[-1]},\hat{x}_{1}))\,.
            \end{aligned}
        \end{equation*}
        Therefore, by the rank condition \eqref{eq:rank_cond_disc} for $k=p-1$, the topmost set of equations
        \begin{equation*}
            \tilde y_p = \tilde y_{p,[-1]} + T_s f_p(\tilde y_p, \tilde y_{p-1}, \hat{x}_{p-1}^+, u_{p-1})
        \end{equation*}
        can be solved for $(\hat{x}_{p-1}^+,u_{p-1})$ as functions of $\tilde y_{p,[-1]}$, $\tilde y_p$, and $\tilde y_{p-1}$. To parameterize the non-shifted state variables $\hat{x}_{p-1}$, an additional backward-shift is required. Consequently, $\hat{x}_{p-1}$ depends on $\tilde y_{p,[-2]}$, $\tilde y_{p,[-1]}$, and $\tilde y_{p-1,[-1]}$. Since the remaining components of $x_{p-1}$ are given by $\tilde y_{p-1,[-1]}$, the variables $(x_{p-1},u_{p-1})$ are completely parameterized by $\tilde y_{p,[-2]}, \tilde y_{p,[-1]}, \tilde y_p, \tilde y_{p-1,[-2]}, \tilde y_{p-1,[-1]}$.
        Proceeding recursively with the remaining blocks of equations ultimately yields a parameterization of the form~\eqref{eq:backward_flat_param}.
    \end{proof}

\section{Controller Design} \label{sec:three}
    In this section, we demonstrate the design of a flatness-based tracking controller for backward-flat systems, which is divided into two steps. First, the system is exactly linearized, and subsequently, an additional feedback is applied to stabilize the reference trajectories.


    \subsection{Exact linearization}
    It was established in \cite{DiwoldKolarSchoberl:2020} that an exact linearization of flat discrete-time systems in the sense of Definition \ref{def:flatness} by an endogenous dynamic feedback is always possible. Since backward-flat systems according to Definition \ref{def:backward-flatness} constitute a subclass of difference flat systems, these results apply directly to the backward-flat case. In particular, for every backward-flat discrete-time system $x^{+}=f(x, u)$, there exists an endogenous dynamic feedback of the form
    \begin{equation}\label{eq:linearizing_feedback_general}
        \begin{aligned}
        z^{+} & =\delta_y(F_z) \circ \hat{\Phi}(x, z, v) \\
        u & =F_u \circ \hat{\Phi}(x, z, v)
        \end{aligned}
    \end{equation}
    with $\operatorname{dim}(z)=\big(\sum_{j=1}^m r_j\big)-n$, such that the resulting closed-loop system
    \begin{equation}\label{eq:sys_dyn_ext}
        \begin{aligned}
        x^{+} & =f\big(x, F_u \circ \hat{\Phi}(x, z, v)\big) \\
        z^{+} & =\delta_y(F_z) \circ \hat{\Phi}(x, z, v)
        \end{aligned}
    \end{equation}
    exhibits the input-output behavior $y_{[0]}=v$, where the map $\Phi(y_{[-R_1,0]})$ is defined as
    \begin{equation*}
        \begin{aligned}
        x & =F_x(y_{[-R_1,-1]}) \\
        z & =F_z(y_{[-R_1,-1]}) \\
        v & =y_{[0]}
        \end{aligned}
    \end{equation*}
    with its inverse $\hat{\Phi}(x, z, v)$ given by
    \begin{equation}\label{eq:Phi_hat}
        \begin{aligned}
        (y_{[-R_1,-1]}) & =\hat{F}_{x z}(x, z) \\
        y_{[0]} &= v \,.
        \end{aligned}
    \end{equation}
    Due to the submersion property of the parameterization $F_x$, there always exists a map $z = F_z(y_{[-R_1,-1]})$, such that the combined mapping $(x, z) = (F_x, F_z) =: F_{xz}$ is a diffeomorphism.
    
    As shown in \cite{SchrotshamerKolarSchoeberl:2025} and \cite{SchrotshamerKolarSchoeberl:2026}, every backward-flat system admits a backward-flat output of the form
    \begin{equation}\label{eq:y_backward}
    y^j=\varphi^j(x,u),
    \qquad j=1,\ldots,m,
    \end{equation}
    that is independent of the variables $\ldots,\zeta_{[-2]},\zeta_{[-1]}$ and satisfies $\operatorname{rank}(\partial_u \varphi_{[0]}) = m$\footnote{The rank condition $\operatorname{rank}(\partial_u \varphi_{[0]}) = m$ must hold in general because substituting the map \eqref{eq:sysEq_ext_inv} into the backward-shifts of \eqref{eq:y_backward} yields functions of $x$ and $\ldots,\zeta_{[-2]},\zeta_{[-1]}$ only. Otherwise, not all components of $u$ could be expressed in terms of $y_{[-R_1,0]}$.}. Consequently, the invertible input transformation in \eqref{eq:linearizing_feedback_general} reduces to
    \begin{equation}\label{eq:linearizing_feedback_reduced}
        u = \hat \varphi_{[0]}(x, v)\,.
    \end{equation}
    
    Moreover, it has recently been shown in \cite{SchrotshamerKolarSchoeberl:2026b} that there exists a coordinate transformation 
    \begin{equation}\label{eq:generalized_trafo}
        \zeta_{[-1]}=\hat{\Phi}_\zeta(x,\bar \zeta_{[-R_1+\lambda-1,-1]})
    \end{equation}
    of the variables $\zeta_{[-1]}$ under which the backward-shifts $y_{[-R_1,0]}$ of the backward-flat output take the form
    \begin{equation}\label{eq:y_backward_transformed}
            \begin{aligned}
            y_{1,[-r_{1,1},-\lambda_1]} &= \bar\zeta_{1,[-r_{1,1}+\lambda_1-1,-1]} \\
            y_{1,[-\lambda_1+1, 0]} &= \varphi_{1,[-\lambda_1+1, 0]}(x,u) \\
            &\hspace{0.2cm}\vdots\\
            y_{s-1,[-r_{1,s-1},-\lambda_{s-1}]} &= \bar\zeta_{s-1,[-r_{1,s-1}+\lambda_{s-1}-1,-1]} \\
            y_{s-1,[-\lambda_{s-1}+1, 0]} &= \varphi_{s-1,[-\lambda_{s-1}+1, 0]}(\bar\zeta_{1,[-r_{1,1}+\lambda_1-1,-1]},\\
            &\ldots,\bar\zeta_{s-2,[-r_{1,s-2}+\lambda_{s-2}-1,-1]},x,u) \\
            y_{s,[-r_{1,s}, -\lambda_s]} &= \bar\zeta_{s,[-r_{1,s}+\lambda_{s}-1,-1]} \\
            y_{s,[-\lambda_s+1, 0]} &= \varphi_{s,[-\lambda_s+1, 0]}(\bar\zeta_{1,[-r_{1,1}+\lambda_1-1,-1]},\\
            &\ldots,\bar\zeta_{s-1,[-r_{1,s-1}+\lambda_{s-1}-1,-1]},x,u)\,,
            \end{aligned}
        \end{equation}
    The transformation \eqref{eq:generalized_trafo} is constructed step by step, over a total of $s$ steps. In the first step, the components of the backward-flat output \eqref{eq:y_backward} are shifted backward until they depend explicitly on the coordinates $\zeta_{[-1]}$, and as many of these functions as possible are introduced as new coordinates $\bar\zeta_{1,[-1]}$ with $\operatorname{dim}(\bar\zeta_{1,[-1]})=m_1$. In each subsequent step $i=2,\ldots,s$, the rest of the output components are shifted further until they depend explicitly on the remaining components of $\zeta_{[-1]}$, and a maximal set of $m_i$ such functions is introduced as new coordinates $\bar\zeta_{i,[-1]}$\footnote{The variables $m_i$, $i=1,\ldots,s$, denote the maximal number of functions that can be normalized in the corresponding step via an invertible coordinate transformation $\zeta_{i,[-1]}=\hat{\Phi}_{\zeta_i}(x,\bar\zeta_{1,[-r_{1,1}+\lambda_1-1,-1]},\ldots,\bar\zeta_{i-1,[-r_{1,i-1}+\lambda_{i-1}-1,-1]},\bar\zeta_{i,[-1]})$. The abbreviation of all transformations carried out in the $s$ steps is then given by \eqref{eq:generalized_trafo}.}. Repeating this procedure until all $m=m_1+\ldots+m_s$ components of $\zeta_{[-1]}$ have been replaced yields a minimal multi-index $\lambda=(\lambda_1,\ldots,\lambda_s)$ such that $\bar\zeta_{[-1]}=\delta^{-\lambda}(\varphi)$.
    Due to this stepwise construction where in every step $m_i$ new coordinates are introduced, it follows that $y_{i}=(y_i^1,\ldots,y_i^{m_i})$, $\varphi_{i}=(\varphi_i^1,\ldots,\varphi_i^{m_i})$, $\zeta_{i}=(\zeta_i^1,\ldots,\zeta_i^{m_i})$, $r_{1,i}=(r_{1,i}^1,\ldots,r_{1,i}^{m_i})$ and $\lambda_{i}=(\lambda_{i}^1,\ldots,\lambda_{i}^{m_i})$.\\    
    In \cite{SchrotshamerKolarSchoeberl:2026b}, it has been shown, by proving that \eqref{eq:y_backward_transformed} is a diffeomorphism, that backward-flat systems can be rendered static feedback linearizable through prelongations of the variables $\bar\zeta_{1,[-r_{1,1}+\lambda_1-1,-1]},\ldots,\bar\zeta_{s-1,[-r_{1,s-1}+\lambda_{s-1}-1,-1]}$. For a detailed derivation of \eqref{eq:y_backward_transformed} and the construction of prelongations that render a backward-flat system static feedback linearizable, we refer to \cite{SchrotshamerKolarSchoeberl:2026b}.\\
    This result, in particular, provides a systematic construction of the dynamic extension \eqref{eq:linearizing_feedback_general} consisting of simple shift-chains. As mentioned in \cite{SchrotshamerKolarSchoeberl:2026b}, the closed-loop system \eqref{eq:sys_dyn_ext} can be transformed into Brunovsky normal form
    \begin{equation}\label{eq:Brunovsky_dyn}
    \begin{array}{rlcrl}
    y^{+}_{1,[-r_{1,1}]} = & \hspace{-0.2cm} y_{1,[-r_{1,1}+1]} & \cdots & y^{+}_{s,[-r_{1,s}]} = & \hspace{-0.2cm} y_{s,[-r_{1,s}+1]} \\
     \vdots \hspace{0.1cm} & & & \vdots \hspace{0.1cm} & \\
    y^{+}_{1,[-1]} = & \hspace{-0.2cm} v_{1} & \cdots & y^{+}_{s,[-1]} = & \hspace{-0.2cm} v_{s}
    \end{array}
    \end{equation}
    with $v_i=(v_i^1,\ldots,v_i^{m_i})$ via the classical state- and input transformation \eqref{eq:Phi_hat}.
    The resulting Brunovsky normal form \eqref{eq:Brunovsky_dyn} consists of shift chains associated with the individual components of the flat output with the corresponding lengths $r_{1,i}^{j_i}$. Consequently, the sum of the lengths of all shift chains is $\#R_1$. As will become relevant for the subsequent tracking controller design, the order of the resulting tracking error dynamics coincides with this sum.
    
    In the following, we show that the dynamic extension can be omitted, which results in shift-chains of total length $\#(\lambda-1)=n\leq\#R_1$ and, in turn, allows for tracking error dynamics of lower order $\#A$, with $A\leq R_1$.
    Since the values of $\bar\zeta_{[-R_1+\lambda-1,-1]}$ are already available from past measurements, the implementation as a dynamic feedback is not necessary -- these values can simply be stored instead. However, if the dynamic extension by the variables $\bar\zeta_{[-R_1+\lambda-1,-1]}$ in the form of prelongations is omitted, the system can no longer be transformed into Brunovsky normal form by means of a classical state- and input transformation.
    Nevertheless, since $\#(\lambda-1)=n$ according to Proposition~1 (iii) in \cite{SchrotshamerKolarSchoeberl:2026b}, the map \eqref{eq:y_backward_transformed} induces the generalized state transformation
    \begin{equation}\label{eq:generalized_state_trafo}
        \tilde{x}=\varphi_{[-\lambda+1,-1]}(\bar\zeta_{[-R_1+\lambda-1,-1]},x)
    \end{equation}
    with $\operatorname{rank}(\partial_x \varphi_{[-\lambda+1,-1]})=n$ and 
    \begin{equation*}
    \begin{aligned}
        \tilde{x}
        ={}&
        \big[
            y^1_{1,[-\lambda_1+1,-1]},
            \ldots,
            y^{m_1}_{1,[-\lambda_1+1,-1]},
            \ldots,\\
            &\hspace{1cm}
            y^1_{s,[-\lambda_s+1,-1]},
            \ldots,
            y^{m_s}_{s,[-\lambda_s+1,-1]}
        \big]^T\,.
    \end{aligned}
    \end{equation*}
    Together with the linearizing feedback \eqref{eq:linearizing_feedback_reduced}, this generalized state transformation brings a backward-flat system \eqref{eq:sysEq} into the form
    \begin{equation}\label{eq:Brunovsky_red}
    \begin{array}{rlcrl}
        y^{+}_{1,[-\lambda_{1}+1]}
        = & \hspace{-0.2cm} y_{1,[-\lambda_{1}+2]}
        & \cdots &
        y^{+}_{s,[-\lambda_{s}+1]}
        = & \hspace{-0.2cm} y_{s,[-\lambda_{s}+2]} \\
        \vdots \hspace{0.1cm} & & & \vdots \hspace{0.1cm} & \\
        y^{+}_{1,[-1]}
        = & \hspace{-0.2cm} v_{1}
        & \cdots &
        y^{+}_{s,[-1]}
        = & \hspace{-0.2cm} v_{s}\,.
    \end{array}
    \end{equation}

    \begin{rem}
        Since the input-output behavior of the closed-loop system is given by $y_{[0]}=v$, the new input $v$ directly determines the values of the flat output. Hence, perfect tracking of a desired trajectory $y^d$ could be achieved by choosing $v=y^d$, resulting in error dynamics of order $0$.
        However, as demonstrated in the following, this approach is not suitable for practical applications. We therefore focus instead on the design of tracking control laws for systems of the forms \eqref{eq:Brunovsky_dyn} and \eqref{eq:Brunovsky_red}, with error dynamics of orders $\#R_1$ and $\#\lambda-m$, respectively.
    \end{rem}
    

    \subsection{Tracking control design}
    In the following, we first demonstrate the design of a tracking control law based on an exact linearization via a dynamic feedback \eqref{eq:linearizing_feedback_general} with a backward-flat output of the form \eqref{eq:y_backward}.
    Subsequently, we consider the case $A\leq R_1$, in which the closed-loop system additionally depends on the backward-shifts $\bar\zeta_{[-R_1+\lambda,-1]}$. However, unlike the design of tracking control laws for forward-flat systems in \cite{DiwoldKolarSchoberl:2022} or flat systems in the sense of Definition \ref{def:flatness} in \cite{KolarDiwoldGstottnerSchoeberl:2022}, no expressions for forward-shifts of $v$ have to be derived. Instead, for backward-flat systems, only values of past measurements and/or inputs $\bar \zeta_{[-R_1+\lambda, -1]}$ need to be stored.

    \subsubsection{$A=R$}
    Consider a backward-flat system \eqref{eq:sysEq} linearized by a dynamic feedback \eqref{eq:linearizing_feedback_general}, where the resulting closed-loop system \eqref{eq:sys_dyn_ext} is subsequently transformed into Brunovsky normal form \eqref{eq:Brunovsky_dyn}.
    The control law
    \begin{equation}\label{eq:tracking_control_law}
        v_i^{j_i}=y_{i,[0]}^{j_i, d}-\sum_{l=-1}^{-r_{1,i}^{j_i}} a_i^{j_i, l}(y_{i,[l]}^{j_i}-y_{i,[l]}^{j_i, d})\,, \qquad j_i=1,\ldots,m_i
    \end{equation}
    asymptotically stabilizes the tracking error $e_i^{j_i}=y_i^{j_i}-y_{i}^{j_i,d}$ with respect to an arbitrary reference trajectory $y_i^{j_i,d}(k)$ by placing the eigenvalues of the linear error dynamics
    \begin{equation*}
    e_{i,[0]}^{j_i}+\sum_{l=-1}^{-r_{1,i}^{j_i}}a_{i}^{j_i,l}e_{i,[l]}^{j_i}=0\label{eq:error_dynamics}
    \end{equation*}
    inside the unit circle.
    Substituting the map \eqref{eq:Phi_hat} into \eqref{eq:tracking_control_law} and inserting the resulting expression into \eqref{eq:linearizing_feedback_general} finally yields a tracking control law of the form
    \begin{equation*}
        \begin{aligned}
            z^{+} & =\alpha(x,z,y^d_{[-R_1,0]})\\
            u & =\beta(x,z,y^d_{[-R_1,0]})\,,
        \end{aligned}
    \end{equation*}
    which depends on the state $x$ of the system \eqref{eq:sysEq}, the controller state $z$ and the reference trajectory $y^d$ and its backward-shifts.

    
    \vspace{1ex}
    \subsubsection{$A=\lambda-1$}
    Instead, if the system equations \eqref{eq:sysEq} are not extended by a dynamic feedback, but rather transformed into the form \eqref{eq:Brunovsky_red} through a generalized state transformation \eqref{eq:generalized_state_trafo} and the input transformation \eqref{eq:linearizing_feedback_reduced}, the control law
    \begin{equation}\label{eq:tracking_control_law_lambda}
        v_i^{j_i}=y_{i,[0]}^{j_i, d}-\sum_{l=-1}^{-\lambda_i^{j_i}+1} a_i^{j_i, l}(y_{i,[l]}^{j_i}-y_{i,[l]}^{j_i, d})\,,
    \end{equation}
    no longer contains tracking-error terms associated with
    $\bar\zeta_{[-R_1+\lambda-1,-1]}$ and
    $y^d_{[-R_1,-\lambda]}$. Consequently, the corresponding linear error dynamics
    \begin{equation}\label{eq:error_dynamics_lambda}
            e_{i,[0]}^{j_i}+\sum_{l=-1}^{-\lambda_i^{j_i}+1} a_i^{j_i, l} e_{i,[l]}^{j_i}=0\,,
    \end{equation}
    is now of order $\#\lambda-m=n$.\\
    Substituting the control law \eqref{eq:tracking_control_law_lambda} into the linearizing feedback \eqref{eq:linearizing_feedback_reduced} yields a tracking control law of the form
    \begin{equation*}
    \begin{aligned}
        u =\beta(x,&\bar\zeta_{1,[-r_{1,1}+\lambda_1,-1]},\ldots,\\
        &\bar\zeta_{s-1,[-r_{1,s-1}+\lambda_{s-1},-1]},y^d_{[-\lambda+1,0]})\,.
    \end{aligned}
    \end{equation*}
    Hence, apart from the state $x$ of the system \eqref{eq:sysEq} and the known reference trajectory $y^d_{[-\lambda+1,0]}$, this tracking control law depends only on past values $\bar\zeta_{1,[-r_{1,1}+\lambda_1,-1]},\ldots,\bar\zeta_{s-1,[-r_{1,s-1}+\lambda_{s-1},-1]}$ of the system trajectory.

    \vspace{1ex}
    \subsubsection{$A<\lambda-1$}
    Next, we consider the case where the multi-index $A$ is chosen even smaller. For $A=0$, all coefficients $a_i^{j_i,l}$ in the control law \eqref{eq:tracking_control_law_lambda} are set to zero. Consequently, all eigenvalues of the tracking error dynamics \eqref{eq:error_dynamics_lambda} are located at the origin of the complex plane, resulting in dead-beat control. Since the input-output behavior of the closed-loop system is given by $y_{[0]}=v$, with $v_i^{j_i}=y_{i,[0]}^{j_i, d}$, perfect tracking is achieved, i.e. $e_i^{j_i}=y_i^{j_i}-y_{i}^{j_i,d}=0$. However, this choice inevitably leads to large control inputs, which may not be realizable due to input constraints.
    For $0<A<\lambda-1$, there still occur eigenvalues located at the origin of the complex plane necessitating large control inputs.
    

\section{Practical Example}  \label{sec:four}
    In this section, we illustrate our results using the practical example of a 2D gantry crane, which has already been studied in \cite{KolarRamsSchlacher:2017,DiwoldKolarSchoberl:2022} in the context of time-optimal flatness-based control and tracking control for forward-flat systems, respectively.
    \\
    Using the flatness-preserving discretization described in Section \ref{subsec:flatness_preserving_disc}, we first derive a backward-flat discrete-time system from the continuous-time model of the 2D gantry crane, which in state representation reads as
    \begin{equation}\label{eq:sys_gantry_crane_cont}
    \begin{array}{rlrl}
    \dot{x}_T \hspace{-0.25cm} & =v_T &\qquad \dot{v}_T \hspace{-0.25cm} & =f_{a_T}(\varphi, \theta, \omega_\theta, F, M) \\
    \dot{\varphi} \hspace{-0.25cm} & =\omega_{\varphi} & \dot{\omega}_{\varphi} \hspace{-0.25cm} & =f_{\alpha_{\varphi}}(\varphi, \theta, \omega_\theta, F, M) \\
    \dot{\theta} \hspace{-0.25cm} & =\omega_\theta & \dot{\omega}_\theta \hspace{-0.25cm} & =f_{\alpha_\theta}(\varphi, \theta, \omega_{\varphi}, \omega_\theta, F, M)
    \end{array}
    \end{equation}
    with the state $x=(x_T, \varphi, \theta, v_T, \omega_{\varphi}, \omega_\theta)$ and the input $u=(F,M)$. It is well known that the system \eqref{eq:sys_gantry_crane_cont} is differentially flat, with the $x$- and $y$-positions of the load serving as the flat output $y=(x_L,y_L)=(x_T-R \varphi \sin (\theta),R \varphi \cos (\theta))$.\\
    As in \cite{DiwoldKolarSchoberl:2022}, the continuous-time system \eqref{eq:sys_gantry_crane_cont} is transformed into structurally flat triangular form
    \begin{equation*}
        \begin{array}{rlrl}
        \dot{x}_L \hspace{-0.25cm} & = v_{L, x} &\qquad \dot{v}_{L, y} \hspace{-0.25cm} & = a_{L, y} \\
        \dot{y}_L \hspace{-0.25cm} & = v_{L, y} & \dot\theta \hspace{-0.25cm} & = \omega_\theta \\
        \dot{v}_{L, x} \hspace{-0.25cm} & = (g-a_{L, y}) \tan (\theta) & \dot{\omega}_\theta \hspace{-0.25cm} & = \alpha_\theta
        \end{array}
    \end{equation*}
    using state- and input transformations $\bar{x}=\Phi_x(x)$ and $\bar{u}=\Phi_u(x,u)$, where $\bar{x}=(x_L,y_L,v_{L,x},v_{L,y},\theta,\omega_\theta)$ denotes the transformed state and $\bar{u}=(a_{L,y},\alpha_\theta)$ the new input.\\
    Applying the implicit Euler discretization yields the sampled-data system
    \begin{equation*}
        \begin{array}{rlrl}
        x_L^{+} \hspace{-0.3cm} & =x_L+T_s v_{L, x}^+ & v_{L, y}^{+} \hspace{-0.3cm} & =v_{L, y}+T_s a_{L, y} \\
        y_L^{+} \hspace{-0.3cm} & =y_L+T_s v_{L, y}^+ & \theta^{+} \hspace{-0.3cm} & =\theta+T_s \omega_\theta^+ \\
        v_{L, x}^{+} \hspace{-0.3cm} & =v_{L, x}+T_s (g-a_{L, y}) \tan (\theta^+) & \omega_\theta^{+} \hspace{-0.3cm} & =\omega_\theta+T_s \alpha_\theta
        \end{array}
    \end{equation*}
    which, according to Theorem 3 in \cite{DiwoldKolarSchoberl:2023}, admits an explicit state representation
    \begin{equation}\label{eq:sampled_data_system_gantry_crane}
        \begin{aligned}
        x_L^{+} & = T_s^2 (g-a_{L, y}) \tan(\alpha_\theta T_s^2+\omega_\theta T_s+\theta)+v_{L, x} T_s+x_L \\
        y_L^{+} & = a_{L, y} T_s^2+v_{L, y} T_s+y_L \\
        v_{L, x}^{+} & = T_s (g-a_{L, y}) \tan(\alpha_\theta T_s^2+\omega_\theta T_s+\theta)+v_{L, x} \\
        v_{L, y}^{+} & = a_{L, y} T_s+v_{L, y} \\
        \theta^{+} & = \alpha_\theta T_s^2+\omega_\theta T_s+\theta \\
        \omega_\theta^{+} & = \alpha_\theta T_s+\omega_\theta
        \end{aligned}
    \end{equation}
    in a neighborhood of the equilibrium $(x_s,u_s)$ for sufficiently small sampling times $T_s>0$. The backward-shifts of the backward-flat output $\tilde y=(x_L^+,y_L^+)=(T_s^2(g- a_{L, y}) \tan ( \alpha_\theta  T_s^2 +\omega_\theta T_s+\theta)+v_{L, x} T_s+x_L,  a_{L, y} T_s^2+v_{L, y} T_s+y_L)$ can be written in the form \eqref{eq:y_backward_transformed}, i.e.
    \begin{equation*}
        \begin{array}{rl}
            y_{1,[-4]} \hspace{-0.2cm}& =\zeta_{1,[-2]} \\
            y_{1,[-3]} \hspace{-0.2cm} & = \zeta_{1,[-1]} \\
            y_{1,[-2]} \hspace{-0.2cm} & = -v_{L, y} T_s+y_L \\
            y_{1,[-1]} \hspace{-0.2cm} & = y_L \\
            y_{1,[0]} \hspace{-0.2cm} & =  a_{L, y} T_s^2+v_{L, y} T_s+y_L \\[2ex]
        \end{array}
    \end{equation*}
    \begin{equation}\label{eq:gantry_crane_y_backward}
        \begin{array}{rl}
            y_{2,[-4]} \hspace{-0.2cm} & = \varphi_{2,[-4]}(\zeta_{1,[-2,-1]},x_L,y_L,v_{L, x},v_{L, y},\theta,\omega_\theta) \\
            y_{2,[-3]} \hspace{-0.2cm} & = \varphi_{2,[-3]}(\zeta_{1,[-1]},x_L,y_L,v_{L, x},v_{L, y},\theta) \\
            y_{2,[-2]} \hspace{-0.2cm} & = -v_{L, x} T_s+x_L \\
            y_{2,[-1]} \hspace{-0.2cm} & = x_L \\
            y_{2,[0]} \hspace{-0.2cm} & = T_s^2(g- a_{L, y}) \tan ( \alpha_\theta  T_s^2 +\omega_\theta T_s+\theta)\\
            & \hspace{4.5cm}+v_{L, x} T_s+x_L
        \end{array}
    \end{equation}
    with $\zeta=(\zeta^1,\zeta^2)=(\zeta_2,\zeta_1)=(\omega_\theta,-v_{L, y} T_s+y_L)$, $\tilde y=(\tilde y^1,\tilde y^2)=(y_2,y_1)$, $\varphi=(\varphi^1,\varphi^2)=(\varphi_2,\varphi_1)$, $R_1=(r_{1,1},r_{1,2})=(4,4)$. Since $y_{2,[-5]}=\varphi_{2,[-5]}(\zeta_{1,[-3,-1]},\zeta_{2,[-1]},x_L,y_L,v_{L,x},v_{L,y},\theta,\omega_\theta)$, the variables $\zeta_{1,[-1]}$ and $\zeta_{2,[-1]}$ first appear in $\varphi_{1,[-3]}$ and $\varphi_{2,[-5]}$, respectively. Hence, the multi-index is given by $\lambda=(\lambda_1,\lambda_2)=(3,5)$.
    The linearizing feedback \eqref{eq:linearizing_feedback_reduced} then reads as
    \begin{equation}\label{eq:gantry_crane_lin_feedback}
    \begin{aligned}
        a_{L,y} & = \frac{v_1-y_L-v_{L,y} T_s}{T_s^2}\\
        \alpha_\theta & = \frac{-\arctan (\frac{x_L+v_{L,x} T_s-v_2}{T_s^2 g+v_{L,y} T_s+y_L-v_1})-\theta-\omega_\theta T_s}{T_s^2}
    \end{aligned}
    \end{equation}
    with the input-output behavior of the closed-loop system given by $y_{1,[0]}=v_1$ and $y_{2,[0]}=v_2$, where $v=(v_1,v_2)$.\\
    As mentioned in Section \ref{sec:three}, the choice $v_1=y_1^d$, $v_2=y_2^d$ yields dead-beat control with error dynamics of order $0$. As can be seen in Figure \ref{fig:Tracking_performance_1}, the initial errors in the variables $x_L$ and $y_L$ are driven to zero in a single step, and the reference trajectory is tracked perfectly thereafter. This, however, comes at the cost of large control inputs $\bar u$.
    \begin{figure}[h!]
	    \noindent
        \begin{centering}
    		\includegraphics[width=0.95\columnwidth]{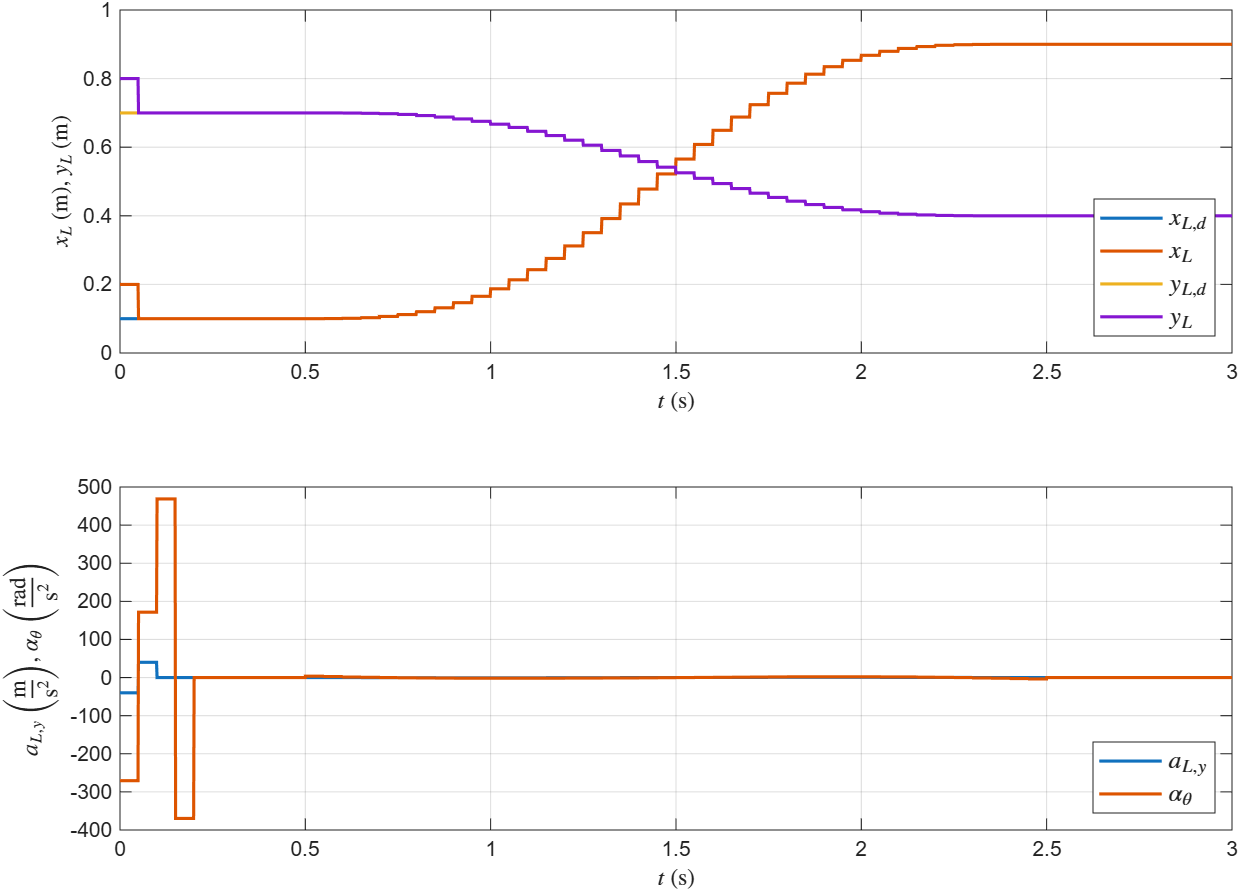}
    		\vspace{-1ex}
    		\caption{Tracking performance for discrete-time system \eqref{eq:sampled_data_system_gantry_crane} with input $\bar u=(a_{L,y},\alpha_\theta)$, sampling time of $T_s=50~\mathrm{ms}$ and $A=0$}
            \label{fig:Tracking_performance_1}
		\par\end{centering}
	\end{figure}
    
    Next, we discuss the case $0<A<\lambda-1$, where eigenvalues of the tracking error dynamics are located at the origin of the complex plane by, e.g., choosing $A=(2,3)$. As illustrated in Figure \ref{fig:Tracking_performance_2}, an appropriate choice of the coefficients $a_i^l$ results in good tracking performance, at the expense of large control inputs.\\
    \begin{figure}[h!]
	    \noindent
        \begin{centering}
    		\includegraphics[width=0.95\columnwidth]{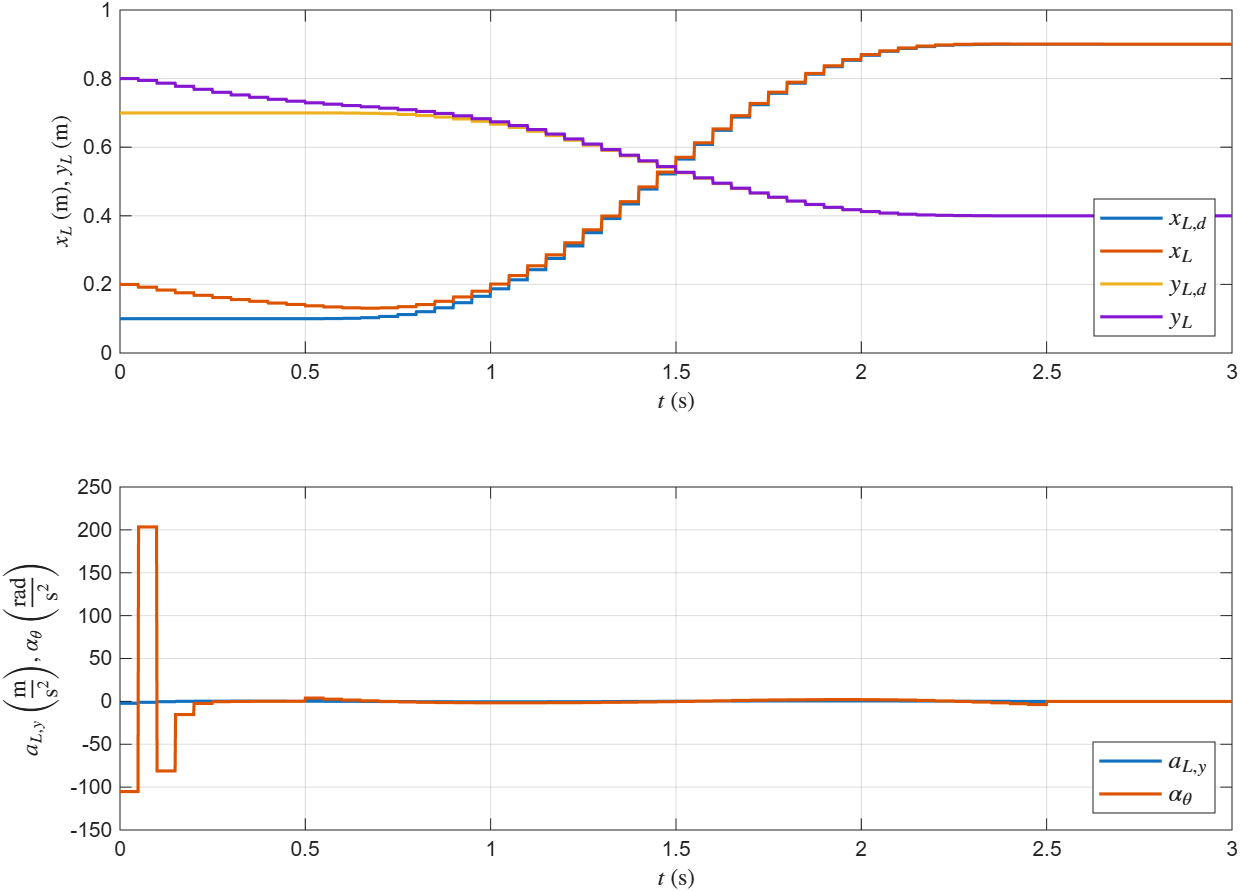}
    		\vspace{-1ex}
    		\caption{Tracking performance for discrete-time system \eqref{eq:sampled_data_system_gantry_crane} with input $\bar u=(a_{L,y},\alpha_\theta)$, sampling time of $T_s=50~\mathrm{ms}$ and $0<A<\lambda-1$}
            \label{fig:Tracking_performance_2}
		\par\end{centering}
	\end{figure}

    Finally, we consider the case $A=\lambda-1$, for which a stabilizing control law is given by 
    \begin{equation}\label{eq:stab_control_law_gantry_crane}
    \begin{aligned}
        v_1 & = y_{1,[0]}^{d}-\sum_{l=-1}^{-2} a_1^{l}(y_{1,[l]}-y_{1,[l]}^{d}) \\
        v_2 & = y_{2,[0]}^{d}-\sum_{l=-1}^{-4} a_2^{l}(y_{2,[l]}-y_{2,[l]}^{d})
    \end{aligned}
    \end{equation}
    with $\lambda=(3,5)$ and suitably chosen coefficients $a^l_i$.
    Substituting the map \eqref{eq:gantry_crane_y_backward} into the control law \eqref{eq:stab_control_law_gantry_crane} and inserting the resulting expression into \eqref{eq:gantry_crane_lin_feedback} yields a tracking control law of the form
    \begin{equation*}
    \begin{aligned}
        a_{L,y} & = \beta_1(y_L,v_{L, y},y^d_{1,[-2,0]})\\
        \alpha_\theta & = \beta_2(x_L,y_L,v_{L, x},v_{L, y},\theta,\omega_\theta, \\
        & \hspace{2cm} \zeta_{1,[-2,-1]},y^d_{1,[-2,0]},y^d_{2,[-4,0]})\,.
    \end{aligned}
    \end{equation*}
    As already mentioned in Section \ref{subsec:flatness_preserving_disc}, these inputs are related to the original inputs of the continuous-time system \eqref{eq:sys_gantry_crane_cont} via the map $\bar u=\Phi_u(x^+,u)$. Thus, the original input can be derived from the inverse mapping, i.e., $u=(F, M)=\hat\Phi_u(y_L^+, v_{L, y}^+, \theta^+, \omega_\theta^+, a_{L, y}, \alpha_\theta)$. The resulting tracking control law achieves the tracking performance illustrated in Figure \ref{fig:Tracking_performance_3} for the continuous-time system \eqref{eq:sys_gantry_crane_cont}, with control inputs that can realistically be provided by appropriately sized actuators.
    \begin{figure}[h!]
	    \noindent
        \begin{centering}
    		\includegraphics[width=0.95\columnwidth]{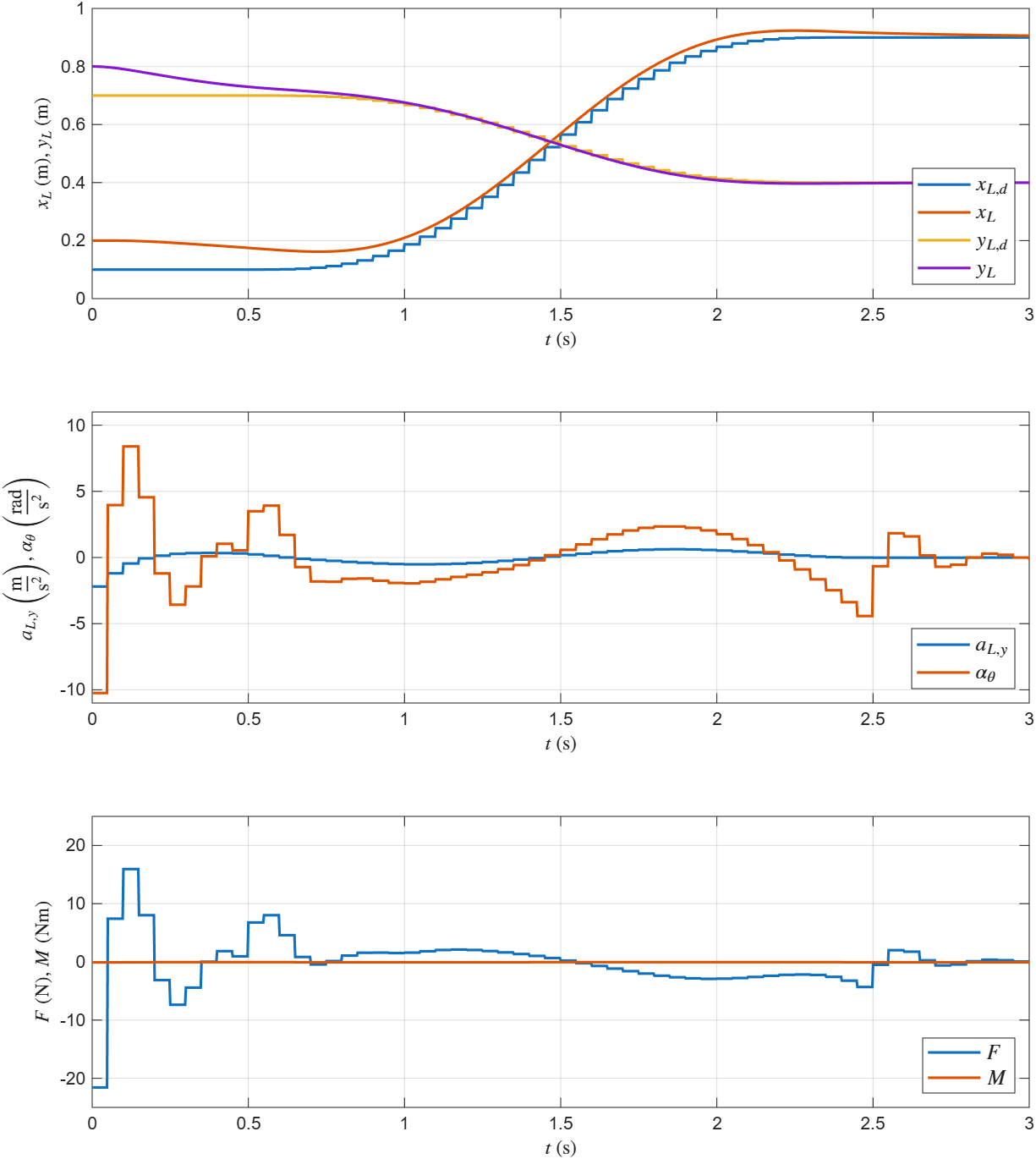}
    		\vspace{-1ex}
    		\caption{Tracking performance for continuous-time system \eqref{eq:sys_gantry_crane_cont} with transformed input $\bar u=(a_{L,y},\alpha_\theta)$, original input $u=(F,M)$, sampling time of $T_s=50~\mathrm{ms}$ and $A=\lambda-1$ }
            \label{fig:Tracking_performance_3}
		\par\end{centering}
	\end{figure}


\bibliographystyle{IEEEtran}

\bibliography{IEEEabrv,bibliography_Johannes}             

                 
\end{document}